\documentclass[a4paper,12pt]{article}
\usepackage[margin=1in]{geometry}
\usepackage{amsmath}
\usepackage{amssymb}
\usepackage{amsthm}
\usepackage{bbold}
\usepackage{hyperref}
\usepackage[version=0.96]{pgf}
\usepackage{tikz}
\usetikzlibrary{arrows,matrix}

\hypersetup{colorlinks=true}
\hypersetup{linkcolor=black,citecolor=black,urlcolor=black,filecolor=black,menucolor=black}

\newtheorem{theorem}{\bf Theorem}
\newtheorem{definition}[theorem]{\bf Definition}
\newtheorem{lemma}[theorem]{\bf Lemma}
\newtheorem{corollary}[theorem]{\bf Corollary}
\newtheorem{remark}[theorem]{\bf Remark}
\newtheorem{example}[theorem]{\bf Example}

\DeclareMathOperator{\id}{id}
\DeclareMathOperator{\supp}{supp}
\DeclareMathOperator{\Hom}{Hom}
\DeclareMathOperator{\im}{im}

\newcommand{\R}{\ensuremath{R}}
\newcommand{\Int}{\ensuremath{{\textstyle\int}}}
\newcommand{\Der}{\ensuremath{\partial}}

\title{Formal proofs of operator identities by a single formal computation}
\author{Clemens G. Raab, Georg Regensburger, and Jamal Hossein Poor}
\date{\small{Inst.\ f.\ Algebra, Johannes Kepler University, Linz, Austria}}

\begin{document}
\maketitle

\begin{abstract}
A formal computation proving a new operator identity from known ones is, in principle, restricted by domains and codomains of linear operators involved, since not any two operators can be added or composed.
Algebraically, identities can be modelled by noncommutative polynomials and such a formal computation proves that the polynomial corresponding to the new identity lies in the ideal generated by the polynomials corresponding to the known identities.
In order to prove an operator identity, however, just proving membership of the polynomial in the ideal is not enough, since the ring of noncommutative polynomials ignores domains and codomains.
We show that it suffices to additionally verify compatibility of this polynomial and of the generators of the ideal with the labelled quiver that encodes which polynomials can be realized as linear operators.
Then, for every consistent representation of such a quiver in a linear category, there exists a computation in the category that proves the corresponding instance of the identity.
Moreover, by assigning the same label to several edges of the quiver, the algebraic framework developed allows to model different versions of an operator by the same indeterminate in the noncommutative polynomials.
\end{abstract}

\section{Introduction}

Very often, properties of matrices or linear operators can be characterized by identities they satisfy. Proofs of new identities corresponding to a claim often are done by formal computations using known identities corresponding to the assumptions. For ensuring validity of such formal computations, one has to check in every step that all expressions are compatible with the formats of the corresponding matrices (resp.\ with the domains and codomains of corresponding linear operators). This means that the sums and products occurring can indeed be formed with matrices (resp.\ linear operators).
In short, the main result of this paper is that, instead of checking every single step of the formal computation, it suffices to check compatibility of the expressions in the new identities to be proven. In particular, if during the computation compatibility was ignored, the algebraic framework we develop ensures the existence of a (possibly different) computation that only uses compatible expressions.
Moreover, an advantage of this framework is that a single formal computation implies the claim for all situations in which the assumptions and claim can be formulated (e.g.\ matrices of different sizes, bounded linear operators on Hilbert spaces, homomorphisms of modules, etc.).

It is standard in algebra to treat identities via polynomial ideals. However, in the context of operators mapping between different spaces, the justification of this approach is more subtle. In the following, we illustrate this by the straightforward proof of a parametrization of inner inverses.
In general, the situation will be more involved than in this illustrative example, of course.
After that, we give a general explanation of our main result. A more formal summary of the notions and results developed in this paper can be found in Section~\ref{sec:summary} below.

An operator $A^-$ is called inner inverse (or g-inverse or $\{1\}$-inverse) of an operator $A$, if $AA^-A=A$. 
Assume, given a linear operator $A$ and an inner inverse $A^-$, we want to prove that also
\[
 A^-+Y-YAA^-
\]
is an inner inverse of $A$ for any linear operator $Y$. Indeed, we can compute
\begin{align*}
 A(A^-+Y-YAA^-)A &= AA^-A+AYA-AYAA^-A\\
 &= A+AYA-AYA = A,
\end{align*}
where the first equality follows from linearity and the second rewrites $AA^-A=A$. Evidently, this computation is valid for linear operators $A,A^-,Y$ on a single space $V$. However, one can check that the same computation also makes sense for linear operators $A:V\to{W}$, $A^-:W\to{V}$, and $Y:W\to{V}$ between different spaces, for example. These situations correspond to the following diagrams.
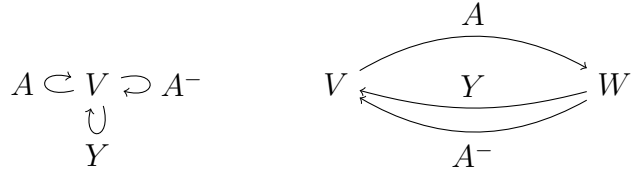
\begin{center}
\begin{tikzpicture}
\begin{scope}
 \node (v) {$V$};
 \path[->] (v) edge [loop left] node [auto] {$A$} (v) edge [loop right] node [auto] {$A^-$} (v) edge [loop below] node [auto] {$Y$} (v);
\end{scope}
\begin{scope}[xshift=5cm]
 \matrix (m) [matrix of math nodes, column sep=3cm]
  {V & W \\};
 \path[->] (m-1-1) edge [bend left] node [auto] {$A$} (m-1-2);
 \path[->] (m-1-2) edge [bend left] node [auto] {$A^-$} (m-1-1);
 \path[->] (m-1-2) edge [bend left=15] node [auto, swap] {$Y$} (m-1-1);
\end{scope}
\end{tikzpicture}
\end{center}

For emphasizing the algebraic aspects of this computation, we write noncommutative indeterminates $a,a^-,y$ for the operators $A,A^-,Y$. Thereby, we translate expressions of operators to noncommutative polynomials in $a,a^-,y$. Noncommutative polynomials arising in this way will be compatible with the situations described by the diagrams above.
Moreover, also identities of operators can be viewed as polynomials in the same indeterminates by forming the difference of their left and right hand sides. In particular, the assumption $AA^-A=A$ used above for rewriting corresponds to the polynomial
\[
 g:=aa^-a-a
\]
and the claim $A(A^-+Y-YAA^-)A=A$ corresponds to
\[
 f:=a(a^-+y-yaa^-)a-a.
\]
Then, the computation above can be rewritten as $f = g-ayg$, where $g$ and $-ayg$ correspond to the rewriting done in the second equality. In particular, this shows that $f$ is contained in the ideal generated by $g$.
Since the operator corresponding to $g$ is zero, the operator corresponding to $f$ is zero as well, if all intermediate steps of the computation are compatible with the domains and codomains of the operators.

In general, assumptions consist of several identities and one considers the ideal generated by the noncommutative polynomials corresponding to these identities.
If an operator identity follows from the assumptions by arithmetic operations with operators (i.e.\ addition, composition, and scaling), then the polynomial corresponding to this identity is contained in the ideal. Conversely, if we know that a polynomial lies in the ideal, does there exist a computation with operators that proves the corresponding operator identity?
An answer to this question is not obvious since, in contrast to computations with actual operators, computations with polynomials are not restricted and all sums and products can be formed. Obviously, elements of the ideal that are not compatible with the diagrams above cannot correspond to identities of operators anyway.
However, our main result states that any element of the ideal corresponds to a true operator identity as soon as it is compatible.
Note that, showing membership in the ideal can be done independently of the operator context. In other words, proving an operator identity is reduced to checking ideal membership of the corresponding polynomial and verifying that this polynomial and the generators of the ideal all are compatible with the diagram describing the domains and codomains of operators.
Hence, once ideal membership is shown, any diagram that is compatible with the polynomials encoding assumptions and claim will automatically give rise to a valid statement about operators.

For commutative polynomials over a field, ideal membership can be decided by Buchberger's algorithm \cite{Buchberger1965} computing a Gr\"obner basis of the ideal.
In contrast, ideal membership of noncommutative polynomials is undecidable in general. Still, a noncommutative analog of Buchberger's algorithm \cite{Mora1994} can be used to verify ideal membership in many cases in practice.
The same can be said about noncommutative polynomials over a finitely generated commutative ring $\R$ with unit element, see \cite{MikhalevZolotykh1998} and references therein.
If a given polynomial can be verified to lie in a given ideal, then the computation of a (partial) Gr\"obner basis can be done in a way that provides an explicit representation of the polynomial in terms of the generators of the ideal. Such a representation can then be checked independently.

Gr\"obner bases for noncommutative polynomials have been applied to operator identities in the pioneering work \cite{HeltonWavrik1994,HeltonStankusWavrik1998}. There, Gr\"obner bases are used to simplify matrix identities in linear systems theory. In~\cite{HeltonStankus1999,Kronewitter2001}, Gr\"obner bases techniques are applied to discover operator identities and to solve matrix equations and matrix completion problems.
It has been observed in the literature that the operations used in the noncommutative analog of Buchberger's algorithm respect compatibility of polynomials with domains and codomains of operators, cf.~\cite[Thm.~25]{HeltonStankusWavrik1998}. For an analogous observation in the context of path algebras, see~\cite[Sec.~47.10]{Mora2016}.
Shortly before initial submission of the present article, we were informed in personal communication that questions related to proving operator identities via computations of Gr\"obner bases are also addressed in \cite{LevandovskyySchmitz2019}.

An alternative approach to modelling computations with operators does not use computations in algebras but in partial algebras arising from diagrams, for which an analogous notion of Gr\"obner bases was sketched in \cite[Sec.~9]{Bergman1978} and developed in \cite{BokutChenLi2012}.
Moreover, using rewriting for operators with domains and codomains, generalizations of Gr\"obner bases and syzygies are considered in \cite{GuiraudHoffbeckMalbos2019}, where higher-dimensional linear rewriting systems are introduced.

\subsection{Summary of the algebraic framework}
\label{sec:summary}

For the convenience of the reader, we outline the main notions and statements of the paper.
From an algebraic point of view, the restrictions imposed above on polynomials in the noncommutative indeterminates $a,a^-,y$ are encoded by the following directed multigraph where edges are labelled by indeterminates, which is also called a labelled quiver.
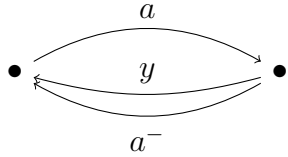
\begin{center}
\begin{tikzpicture}
\begin{scope}
 \matrix (m) [matrix of math nodes, column sep=3cm]
  {\bullet & \bullet \\};
 \path[->] (m-1-1) edge [bend left] node [auto] {$a$} (m-1-2);
 \path[->] (m-1-2) edge [bend left] node [auto] {$a^-$} (m-1-1);
 \path[->] (m-1-2) edge [bend left=15] node [auto, swap] {$y$} (m-1-1);
\end{scope}
\end{tikzpicture}
\end{center}
In order to translate compatible polynomials back to operators, we assign spaces $V$ and $W$ to the vertices of the quiver and to the edges we assign operators $A,A^-,Y$ mapping between corresponding spaces. Such an assignment is called a representation of the quiver. Both situations above are representations of this quiver, depending on whether the spaces $V$ and $W$ are the same.
For each representation of the quiver, a compatible polynomial gives rise to a realization as an operator by replacing the indeterminates by the operators assigned to them.

In general, operators are viewed as $K$-linear maps over a field $K$, e.g.\ $K=\mathbb{R}$ or $K=\mathbb{C}$ for bounded linear operators. The corresponding indeterminates are collected in a set $X$ and computations are done in the free algebra $K\langle{X}\rangle$ of noncommutative polynomials with coefficients in $K$.
The diagram describing this situation is formalized as a \emph{labelled quiver} $Q$, which simply is a directed multigraph where edges have labels in $X$. Then, composition of operators corresponds to paths in $Q$ and sums can only be formed when each path of the summands has same start and same end. A polynomial in $K\langle{X}\rangle$ is called \emph{compatible} with $Q$ if all its monomials correspond to paths in $Q$ with the same start and same end.

The known identities satisfied by the operators in question are translated to a set $F \subseteq K\langle{X}\rangle$. Therefore, operators corresponding to elements of $F$ are zero. For this set $F$, we consider the (two-sided) ideal $(F)$ in $K\langle{X}\rangle$, which is given by all polynomials of the form $f=\sum_ia_if_ib_i$ with $a_i,b_i \in K\langle{X}\rangle$ and $f_i \in F$.
We call a polynomial $f \in (F)$ a \emph{$Q$-consequence} of $F$ if it can be obtained from $F$ by doing only computations with polynomials that are compatible with the labelled quiver $Q$. As will be shown later, this means that the operator corresponding to $f$ is obtained by a valid computation with operators and hence is zero as well, i.e.\ the operator identity corresponding to $f$ holds.
One key point in our approach is to characterize those elements of the ideal $(F)$ that are $Q$-consequences of $F$.
The following characterization arises later as Corollary~\ref{cor:consequences}. If $f \in (F)$ and $Q$ is such that the elements of $F$ are compatible with $Q$, then
 \[
  f\text{ is compatible with }Q \quad\Longleftrightarrow\quad f\text{ is a $Q$-consequence of }F.
 \]

In order to rigorously obtain statements about operators from statements about polynomials in $K\langle{X}\rangle$, we assign a $K$-vector space to each vertex of $Q$ and to each edge of $Q$ we assign a $K$-linear map between corresponding spaces. The family $\mathcal{V}$ of vector spaces together with the assignment $\varphi$ of linear maps is called a \emph{representation} of $Q$, c.f.\ \cite{DerksenWeyman2005}.
For polynomials that are compatible with $Q$, we obtain \emph{realizations} of them as $K$-linear maps via such a representation $(\mathcal{V},\varphi)$ in the way indicated above.
A version of the main result of this paper can then be stated as follows. A more general statement is proven later in Theorem~\ref{thm:Klinear} with the special case explained in Remarks~\ref{rem:compatibilityUnique} and \ref{rem:consistencyUnique}.
The proof relies on the above characterization of $Q$-consequences and on the fact that realizations of $Q$-consequences can be expressed, as $K$-linear maps, in terms of realizations of elements of $F$.
Note that the notion of $Q$-consequences is only used in the proof of the main result and does not play a role in its statement.
In order to avoid constant terms in polynomials as assumed by the theorem below, identity operators occurring have to be assigned their own indeterminates and properties of composition with other operators have to be included in the assumptions.

\begin{theorem}\label{thm:KlinearUnique}
 Let $K$ be a field, let $F \subseteq K\langle{X}\rangle$ such that the polynomials in $F$ do not have a constant term, and let $f \in (F)$. Moreover, let $Q$ be a labelled quiver where edges have unique labels in $X$ and assume that $f$ and all elements of $F$ are compatible with $Q$. Then, for all representations $\mathcal{R}=(\mathcal{V},\varphi)$ of $Q$ such that the realizations of the elements of $F$ w.r.t.\ $\mathcal{R}$ are zero, we have that the realization of $f$ w.r.t.\ $\mathcal{R}$ is zero.
\end{theorem}

Based on this theorem, the sufficient formal computation for proving claimed properties of operators with fixed domains and codomains is given by the following three steps.
\begin{enumerate}
 \item All the assumptions on the operators involved have to be phrased in terms of identities involving those operators. Likewise, the claimed properties have to be expressed as identities of these operators.
 \item These identities are converted into polynomials by uniformly replacing the individual operators by noncommutative indeterminates from some set $X$ in the differences of the left and right hand sides.
 \item Prove that the polynomials corresponding to the claim lie in the ideal generated by the set $F$ of polynomials corresponding to the assumptions.
\end{enumerate}
Then, our main result shows that the assumptions on the operators involved imply the claimed identities. This relies on the important fact that all polynomials arising from identities of operators are automatically compatible with the quiver $Q$ describing the situation of operators.
Just by changing the representation $(\mathcal{V},\varphi)$ of $Q$, we even obtain analogous statements about other types of operators in a rigorous way without doing any additional computations.
See Example~\ref{ex:WoodburyInner} for a worked example.

In practice, ideal membership $f \in (F)$ can often be shown with the help of the computer via computing a (partial) Gr\"obner basis of $(F)$, even though the problem is undecidable in general. Checking compatibility of polynomials can be automatized as well.
The \textsc{Mathematica} package \texttt{OperatorGB} provides functionality for both tasks and it also provides an explicit representation of $f$ in terms of the generators of the ideal, see \cite{HofstadlerRaabRegensburger2019}. The package is available at \url{http://gregensburger.com/softw/OperatorGB} along with documentation and examples.
An implementation in \textsc{SageMath} is also available from the same webpage.

The algebraic framework and the main results can be generalized in various ways.
For instance, the field $K$ can be replaced by a commutative ring $\R$ with unit element if at the same time vector spaces are replaced by $\R$-modules.
Moreover, in some situations, it might be possible to model some of the operators involved by the same indeterminate, because they satisfy the same family of identities. This mainly happens if those operators actually are versions of the same operator with different domains and codomains, e.g.\ differential and integral operators can be interpreted on different domains. Example~\ref{ex:ODEfact} is used to illustrate this. The advantage of using the same indeterminate to represent different operators is that computations with polynomials usually are simpler.
Even in the trivial example above, the operator $A$ could exist in two different versions that have different domains $A_1:U\to{W}$ and $A_2:V\to{W}$. The following quiver and representation describe this situation.
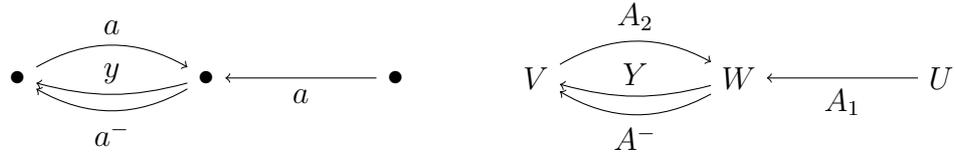
\begin{center}
\begin{tikzpicture}
\begin{scope}
 \matrix (m) [matrix of math nodes, column sep=2cm]
  {\bullet & \bullet & \bullet\\};
 \path[->] (m-1-3) edge node [auto] {$a$} (m-1-2);
 \path[->] (m-1-1) edge [bend left] node [auto] {$a$} (m-1-2);
 \path[->] (m-1-2) edge [bend left] node [auto] {$a^-$} (m-1-1);
 \path[->] (m-1-2) edge [bend left=15] node [auto, swap] {$y$} (m-1-1);
\end{scope}
\begin{scope}[xshift=7cm]
 \matrix (m) [matrix of math nodes, column sep=2cm]
  {V & W & U\\};
 \path[->] (m-1-3) edge node [auto] {$A_1$} (m-1-2);
 \path[->] (m-1-1) edge [bend left] node [auto] {$A_2$} (m-1-2);
 \path[->] (m-1-2) edge [bend left] node [auto] {$A^-$} (m-1-1);
 \path[->] (m-1-2) edge [bend left=15] node [auto, swap] {$Y$} (m-1-1);
\end{scope}
\end{tikzpicture}
\end{center}
For instance, $V$ might be a subspace of $U$ and $A_2$ the restriction of $A_1$ on $V$. So, both $A_1$ and $A_2$ are modelled by the indeterminate $a$ in the algebraic computation with polynomials $f$ and $g$ above. Then, the fact that both realizations of the polynomial $g$ are zero, i.e.\ $A_2A^-A_1=A_1$ and $A_2A^-A_2=A_2$, implies that both realizations of $f$ are zero, i.e.\ $A_2(A^-+Y-YA_2A^-)A_1=A_1$ and $A_2(A^-+Y-YA_2A^-)A_2=A_2$.

We also extend our algebraic framework to treat such cases rigorously, which makes some definitions and auxiliary results more complicated.
In particular, compatibility of polynomials with a labelled quiver can be generalized in different ways.
For the polynomial corresponding to the claim it suffices that, for all its monomials, there exist corresponding paths with the same start and same end. For each of the polynomials corresponding to the assumptions, however, we need to impose the stronger condition that the possibilities for start and end of corresponding paths are uniform for all its monomials.

The remainder of this paper gives formal definitions and detailed proofs for making the general form of the framework completely rigorous. The introductory example above will be used to illustrate the notions introduced and we also give more worked out examples to illustrate the use of the algebraic framework.
We explain rewriting of polynomials in $\R\langle{X}\rangle$ and its relation to ideal membership in Section~\ref{sec:rewriting}.
In Section~\ref{sec:compatibility}, we will give a formal definition of compatibility and uniform of polynomials with a quiver that allows several edges of the quiver to have the same label. We also discuss the relation of computations in $\R\langle{X}\rangle$ with compatibility, including the key observation Lemma~\ref{lem:inclusion}.
The notion of $Q$-consequences is formalized and discussed in Section~\ref{sec:Qconsequences}.
Finally, in Section~\ref{sec:RlinearOperators}, the implications for proving operator identities are laid out in detail arriving at our main result Theorem~\ref{thm:Klinear} for $K$-linear maps on $K$-vector spaces. Then, using the more abstract language of categories, a more general version of the main result is given as Theorem~\ref{thm:Rlinear}, where operators are morphisms in $\R$-linear categories.

\section{Rewriting in the free algebra}
\label{sec:rewriting}

For a set $X$, we denote the word monoid with alphabet $X$ by $\langle{X}\rangle$, where multiplication is given by concatenation of words. It is the free monoid on $X$.
For a commutative ring $\R$ with unit element, we recall the free algebra $\R\langle{X}\rangle$ over $\R$ on $X$. It can be regarded as the ring of noncommutative polynomials in the set of indeterminates $X$ with coefficients in $\R$, where indeterminates commute with coefficients but not with each other. The monomials are words $x_1\dots x_n \in \langle{X}\rangle$, $x_i \in X$, including the empty word $1$. Every polynomial $f \in \R\langle{X}\rangle$ has a unique representation as a sum
\[
 f=\sum_{m\in\langle{X}\rangle}c_mm
\]
with coefficients $c_m \in \R$, such that only finitely many coefficients are nonzero, and its support is defined as
\[
 \supp(f):=\{m\in\langle{X}\rangle\ |\ c_m\neq0\},
\]
where $c_m$ are as above. In short, $\R\langle{X}\rangle$ is the monoid ring on $\langle{X}\rangle$ over $\R$.
In what follows, we fix a commutative ring $\R$ with unit element.

When formally rewriting an expression for an operator using known identities, the steps taken can be formalized with noncommutative polynomials as follows. Proving an identity of two expressions amounts to rewriting their difference to zero.

\begin{definition}
 Let $f,g \in \R\langle{X}\rangle$ such that some monomial $m_g \in \supp(g)$ divides some monomial $m_f \in \supp(f)$, i.e.\ $m_f=am_gb$ for some monomials $a,b \in \langle{X}\rangle$.
 For every $\lambda \in \R$, we say that
 \[
  f+\lambda agb \in \R\langle{X}\rangle
 \]
 can be obtained from $f$ by a rewriting step using $g$.
 Furthermore, let $G \subseteq \R\langle{X}\rangle$ and $h \in \R\langle{X}\rangle$. We say that $f$ \emph{can be rewritten to $h$ using $G$}, if there are $f_0,\dots,f_n \in \R\langle{X}\rangle$, $f_0=f$, $f_n=h$, and $g_1,\dots,g_n \in G$ such that, for all $i \in \{1,\dots,n\}$, $f_i$ can be obtained from $f_{i-1}$ by a rewriting step using $g_i$.
\end{definition}

\begin{example}
 We consider the noncommutative polynomial ring $\mathbb{Z}\langle{X}\rangle$ with indeterminates $X=\{a,a^-,y\}$ and the two polynomials $g = aa^-a-a$ and $f = a(a^-+y-yaa^-)a-a$ from the introductory example.
 We check that $f = aa^-a+aya-ayaa^-a-a$ can be rewritten to zero using $\{g\}$. The first monomial of $f$ is the same as the first monomial $m_g := aa^-a$ of $g$. If we choose $\lambda=-1$, we obtain $h := f+\lambda g = aya-ayaa^-a$ with one rewriting step using $g$. The second monomial of $h$ can be written as $aym_g$, so we obtain $h + ayg = 0$ with one more rewriting step using $g$.
\end{example}

The above notion of rewriting steps for (noncommutative) polynomials is a rather weak one. A rewriting step does not necessarily reduce a polynomial $f$ to a {\lq\lq}simpler{\rq\rq} one in any sense, not even does it necessarily reduce or simplify the coefficient of the monomial $m_f$, which is acted on, in any directed way.
Still, once appropriate monomials $m_f,m_g,a,b$ are selected, one often chooses $\lambda$ such that $m_f$ is removed from the support of $f$, if possible. This in turn allows to obtain the following crucial equivalence.

\begin{lemma}\label{lem:idealmembership}
 Let $F \subseteq \R\langle{X}\rangle$ and $f \in \R\langle{X}\rangle$. Then, $f$ lies in the ideal $(F)$ if and only if $f$ can be rewritten to zero using $F$.
\end{lemma}
\begin{proof}
 Note that any sum of the form $\sum_{i=1}^n\lambda_ia_if_ib_i$ with $\lambda_i \in R$, $a_i,b_i \in \langle{X}\rangle$, and $f_i \in F$ either is zero or we can remove a summand $\lambda_ja_jf_jb_j$ from it by a rewriting step using $f_j$, where $j \in \{1,\dots,n\}$ is such that $\supp(\sum_{i=1}^n\lambda_ia_if_ib_i) \cap \supp(\lambda_ja_jf_jb_j) \neq \emptyset$. This implies inductively that any sum of the form $\sum_{i=1}^n\lambda_ia_if_ib_i$ with $\lambda_i \in R$, $a_i,b_i \in \langle{X}\rangle$, and $f_i \in F$ can be rewritten to zero using $F$. Since any $f \in (F)$ has a representation as a sum of that form, every element of $(F)$ can be rewritten to zero using $F$.\par
 For the converse, we observe that the definition of a rewriting step implies $f-h \in (F)$, if $h \in \R\langle{X}\rangle$ can be obtained from $f \in \R\langle{X}\rangle$ by a rewriting step using some $g \in F$. Now, if $f \in \R\langle{X}\rangle$ can be rewritten to zero using $F$, there are $h_0,\dots,h_n \in \R\langle{X}\rangle$, with $h_0=f$ and $h_n=0$ such that, for all $i \in \{1,\dots,n\}$, $h_i$ can be obtained from $h_{i-1}$ by a rewriting step using some $g_i \in F$. Therefore, $f=h_0-h_n=\sum_{i=1}^nh_{i-1}-h_i$ lies in $(F)$.
\end{proof}

\section{Compatible polynomials}
\label{sec:compatibility}

Recall that a tuple $(V,E,s,t)$ with $s,t : E\to{V}$ is called a \emph{quiver} with vertices $V$ and edges $E$. For each edge $e \in E$, the vertices $s(e)$ and $t(e)$ are called its \emph{source} and \emph{target}.
For brevity, we call a quiver with labelled edges a \emph{labelled quiver}.
A labelled quiver with labels $X$ is a tuple $Q=(V,E,X,s,t,l)$ with $s,t : E\to{V}$ and $l:E\to{X}$, where $l(e)$ gives the label of an edge $e \in E$.

In what follows, we fix a labelled quiver $Q=(V,E,X,s,t,l)$.
Based on the labels of edges, we define labels of paths so that concatenation of paths corresponds to multiplication of labels in $\langle{X}\rangle$.
For a nonempty path $p=e_n{\dots}e_1$ in $Q$, i.e.\ $e_i \in E$, we define its label as the monomial
\[
 l(p) := l(e_n)\dots{l(e_1)} \in \langle{X}\rangle
\]
and its source and target as $s(p):=s(e_1)$ and $t(p):=t(e_n)$.
For every vertex $v \in V$, there is a distinct empty path $\epsilon_v$ that starts and ends in $v$ without passing through any edge in between. Its label is defined by
\[
 l(\epsilon_v) := 1 \in \langle{X}\rangle
\]
and its source and target are given by $s(\epsilon_v):=v$ and $t(\epsilon_v):=v$.

\begin{definition}
 For any monomial $m \in \langle{X}\rangle$, we define the set of \emph{signatures} of $m$ as
 \[
  \sigma(m):=\{(s(p),t(p))\ |\ p\text{ a path in $Q$ with }l(p)=m\} \subseteq V\times{V}.
 \]
 For a polynomial $f \in \R\langle{X}\rangle$, we define its set of signatures by
 \[
  \sigma(f):=\bigcap_{m\in\supp(f)}\sigma(m) \subseteq V\times{V}.
 \]
\end{definition}

Note that we have $\sigma(0)=V\times{V}$ and $\sigma(1)=\{(v,v)\ |\ v \in V\}$.
In the special case when edges have unique labels, any nonconstant polynomial obviously has at most one signature.
The sets of sources and targets of a polynomial $f$ are given by the projections of the set of signatures $\sigma(f)$, i.e.\ for the set of sources of $f$ we have $s(f)=\{v \in V\ |\ \exists{w \in V}:(v,w) \in \sigma(f)\}$ and analogously for $t(f)$.

\begin{definition}
 A polynomial $f \in \R\langle{X}\rangle$ is said to be \emph{compatible} with the quiver $Q$, if $\sigma(f)\neq\emptyset$.
 Moreover, for $v,w \in V$, we define the set
 \[
  \label{def:module}
  \R\langle{X}\rangle_{v,w} := \{f \in \R\langle{X}\rangle\ |\ (v,w) \in \sigma(f)\}.
 \]
\end{definition}

In particular, a monomial $m \in \langle{X}\rangle$ is compatible with $Q$ if and only if there is a path $p$ in $Q$ with $l(p)=m$.
For every $v,w \in V$ we have $0 \in \R\langle{X}\rangle_{v,w}$.
Also note that for each $f \in \R\langle{X}\rangle_{v,w}$ there are $n \in \mathbb{N}$, $c_i\in\R$, and paths $p_i$ in $Q$ from $v$ to $w$ such that
\[
 f=\sum_{i=1}^nc_il(p_i).
\]
Consequently, $\R\langle{X}\rangle_{v,w}$ is an $\R$-module and the set of polynomials in $\R\langle{X}\rangle$ compatible with $Q$ is given by $\bigcup_{v,w \in V}\R\langle{X}\rangle_{v,w}$.

\begin{remark}\label{rem:compatibilityUnique}
 If, as in Theorem~\ref{thm:KlinearUnique}, edges have unique labels and only polynomials without constant term are considered, then such polynomials have at most one signature. Therefore, compatible polynomials are automatically uniformly compatible, as defined below.
 Under these assumptions, the sum of two nonzero compatible polynomials is compatible if and only if they have the same signature. Moreover, if $f$ has signature $(v_1,w)$, $g$ has signature $(u,v_2)$, and $fg\neq0$, then the product is compatible if and only if $v_1=v_2$, in which case it has $(u,w)$ as its signature.
 The reader who is only interested in these simplifying assumptions of Theorem~\ref{thm:KlinearUnique} can skip the following two technical lemmas and continue with Lemma~\ref{lem:inclusion}.
\end{remark}

\begin{definition}
 A polynomial $f \in \R\langle{X}\rangle$ that is compatible with the quiver $Q$ is called \emph{uniformly compatible} with $Q$ if all monomials $m\in\supp(f)$ have the same set of signatures $\sigma(m)$.
 Moreover, we define
 \[
  \R\langle{X}\rangle_Q:=\{f \in \R\langle{X}\rangle\ |\ f\text{ is uniformly compatible with }Q\}.
 \]
\end{definition}

Note that any compatible monomial $m \in \langle{X}\rangle$ is automatically uniformly compatible with $Q$.
By definition, a polynomial containing a constant term is only uniformly compatible with $Q$ if all its monomials have the set of signatures $\sigma(1)=\{(v,v)\ |\ v \in V\}$.
If the quiver has only one vertex and for all $x \in X$ there is an edge (i.e.\ loop) with the label $x$, then $\R\langle{X}\rangle_Q=R\langle{X}\rangle$.

If two labelled quivers are isomorphic as quivers, i.e.\ they are equal up to renaming vertices and edges, and corresponding edges have the same labels, then the corresponding sets of compatible polynomials are the same.
Consequently, when drawing a labelled quiver, it suffices to show the labels of the edges.

\begin{example}
\label{ex:compatible}
 We consider the following labelled quivers with $X=\{a,a^-,y\}$ for the diagrams of the introductory example.
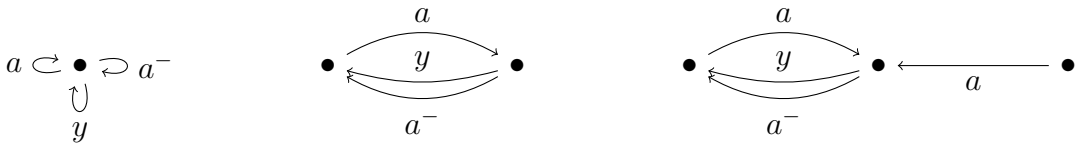
\begin{center}
\begin{tikzpicture}
\begin{scope}[xshift=-4.5cm]
 \node (v) {$\bullet$};
 \path[->] (v) edge [loop left] node [auto] {$a$} (v) edge [loop right] node [auto] {$a^-$} (v) edge [loop below] node [auto] {$y$} (v);
\end{scope}
\begin{scope}
 \matrix (m) [matrix of math nodes, column sep=2cm]
  {\bullet & \bullet \\};
 \path[->] (m-1-1) edge [bend left] node [auto] {$a$} (m-1-2);
 \path[->] (m-1-2) edge [bend left] node [auto] {$a^-$} (m-1-1);
 \path[->] (m-1-2) edge [bend left=15] node [auto, swap] {$y$} (m-1-1);
\end{scope}
\begin{scope}[xshift=6cm]
 \matrix (m) [matrix of math nodes, column sep=2cm]
  {\bullet & \bullet & \bullet\\};
 \path[->] (m-1-3) edge node [auto] {$a$} (m-1-2);
 \path[->] (m-1-1) edge [bend left] node [auto] {$a$} (m-1-2);
 \path[->] (m-1-2) edge [bend left] node [auto] {$a^-$} (m-1-1);
 \path[->] (m-1-2) edge [bend left=15] node [auto, swap] {$y$} (m-1-1);
\end{scope}
\end{tikzpicture}
\end{center}
 We check that the two polynomials $g = aa^-a-a$ and $f = a(a^-+y-yaa^-)a-a$ in $\mathbb{Z}\langle{X}\rangle$, which were considered in the introductory example, are compatible with each of those quivers.
 The quiver on the left has only one vertex, so all polynomials in $\mathbb{Z}\langle{X}\rangle$ are (uniformly) compatible with it, since it contains a loop for each element of $X$.
 For the quiver in the middle, we easily check that all monomials occurring in $g$ and $f = aa^-a+aya-ayaa^-a-a$ are compatible with it and that they have the same set of signatures, which contains a unique element in this case. So, both of these polynomials are uniformly compatible with this quiver.
 The last quiver has two edges with the label $a$. Still, all monomials in $g$ and $f$ have the same set of signatures, which now contains two elements, making both uniformly compatible with that quiver.
\end{example}

In general, the set of compatible polynomials and the set $\R\langle{X}\rangle_Q$ of uniformly compatible polynomials are not closed under addition. However, based on the definition, the following facts are easy to see.

\begin{lemma}\label{lem:addition}
 Let $f,g \in \R\langle{X}\rangle$ be compatible with $Q$ such that $\sigma(f)\cap\sigma(g)\neq\emptyset$. Then, $f+g$ is compatible with $Q$ and $\sigma(f+g) \supseteq \sigma(f)\cap\sigma(g)$. Moreover, if $f,g \in \R\langle{X}\rangle_Q$ with $\sigma(f)=\sigma(g)$, then $f+g \in \R\langle{X}\rangle_Q$.
\end{lemma}

Likewise, the set of compatible polynomials and the set $\R\langle{X}\rangle_Q$ are not closed under multiplication in general, they are only closed under scalar multiplication by elements of $\R$.
However, if $s(f) \cap t(g) \neq \emptyset$, the product $fg$ is again compatible resp.\ uniformly compatible with $Q$, as the following lemma shows.

\begin{lemma}\label{lem:multiplication}
 Let $f,g \in \R\langle{X}\rangle$ be compatible with $Q$ such that $s(f) \cap t(g) \neq \emptyset$.
 Then, $fg$ is compatible with $Q$ and we have
 \[
  \sigma(fg)\supseteq\{(u,w) \in s(g)\times{t(f)}\ |\ \exists{v \in s(f) \cap t(g)}: (u,v) \in \sigma(g) \wedge (v,w) \in \sigma(f)\}.
 \]
 If in addition $f$ and $g$ are uniformly compatible with $Q$ and $fg\neq0$, then equality holds and $fg$ is uniformly compatible as well.
\end{lemma}
\begin{proof}
 It suffices to show the statement for the case of monomials, the general case then follows in a straightforward way.
 Let $m_f, m_g \in \langle{X}\rangle$ be compatible with $Q$, and assume $f=m_f$ and $g=m_g$.\par
 Let $(u,w) \in \sigma(m_fm_g)$, then there exists a path $p$ in $Q$ from $u$ to $w$ with $l(p)=m_fm_g$. This path can be split into two parts $p_1$ and $p_2$ with $l(p_1)=m_f$ and $l(p_2)=m_g$. In particular, $u=s(p_2)$, $s(p_1)=t(p_2)$, and $w=t(p_1)$. Hence, with $v:=s(p_1)$ we have that $(u,v) \in \sigma(m_g)$ and $(v,w) \in \sigma(m_f)$.\par
 Conversely, let $u,v,w \in V$ such that $(u,v) \in \sigma(m_g)$ and $(v,w) \in \sigma(m_f)$. Then, in $Q$, there are paths $p_1$ from $v$ to $w$ and $p_2$ from $u$ to $v$ such that $l(p_1)=m_f$ and $l(p_2)=m_g$. Their concatenation $p_1p_2$ is again a path in $Q$. It has source $u$, target $w$, and label $m_fm_g$. Therefore, $(u,w) \in \sigma(m_fm_g)$.
\end{proof}

\begin{remark}\label{rem:factorization}
 Equality for $\sigma(fg)$ in the above lemma implies that for every $u,w \in V$ such that $fg \in \R\langle{X}\rangle_{u,w}$ there exists $v \in V$ on some path in $Q$ from $u$ to $w$ such that $f \in \R\langle{X}\rangle_{v,w}$ and $g \in \R\langle{X}\rangle_{u,v}$. In particular, this holds if $f$ and $g$ are uniformly compatible with $Q$.
\end{remark}

If the sum or the product of uniformly compatible polynomials are again uniformly compatible, then we have the following converse of Lemma~\ref{lem:addition} and \ref{lem:multiplication}.
\begin{lemma}\label{lem:uniform}
 Let $f,g \in \R\langle{X}\rangle_Q$ be nonzero. If $f+g \in \R\langle{X}\rangle_Q$, then $\sigma(f)=\sigma(g)$. If $fg \in \R\langle{X}\rangle_Q$ and $fg\neq0$, then $s(f) \cap t(g) \neq \emptyset$ and equality for $\sigma(fg)$ holds in the above lemma.
\end{lemma}

The key observation for our approach, however, is that the set of compatible polynomials is closed under rewriting steps by uniformly compatible polynomials.

\begin{lemma}\label{lem:inclusion}
 Let $f,g,h \in \R\langle{X}\rangle$ such that $f$ is compatible, $g$ is uniformly compatible, and $h$ can be obtained from $f$ by a rewriting step using $g$, i.e.
 \[
  h=f+\lambda agb
 \]
 for some $\lambda \in \R$ and $a,b \in \langle{X}\rangle$. Then, $h$ is compatible, $a,b,agb$ are uniformly compatible, and we have $\sigma(h) \supseteq \sigma(f)$ and $\sigma(agb) \supseteq \sigma(f)$.
 Moreover, if in addition $f$ is uniformly compatible, then also $h$ is uniformly compatible and $\sigma(agb) = \sigma(f)$.
\end{lemma}
\begin{proof}
 Let $m_f \in \supp(f)$ and $m_g \in \supp(g)$ such that $m_f=am_gb$ and $h=f+\lambda agb$.
 Note that $m_f$ and $m_g$ are compatible with $Q$ and satisfy $\sigma(f)\subseteq\sigma(m_f)$ and $\sigma(g)=\sigma(m_g)$, respectively. Since $m_f=am_gb$ is a compatible monomial, its factors $a$ and $b$ evidently are uniformly compatible with $Q$ as well and both $s(a) \cap t(m_g)$ and $s(am_g) \cap t(b)$ are nonempty. Hence, $agb \in \R\langle{X}\rangle_Q$ by Lemma~\ref{lem:multiplication} and we have $\sigma(agb)=\sigma(am_gb)=\sigma(m_f)$.
 Consequently, $\sigma(agb) \supseteq \sigma(f)$ and hence $\sigma(h) \supseteq \sigma(f)$ by Lemma~\ref{lem:addition}.
\end{proof}

\section{$Q$-consequences}
\label{sec:Qconsequences}

A (two-sided) ideal in $\R\langle{X}\rangle$ generated by polynomials that are (uniformly) compatible with $Q$ in general will also contain polynomials that are not compatible with $Q$. This leads to the following definition of polynomials that can be obtained from polynomials that are uniformly compatible with $Q$ by doing computations only inside the set of compatible polynomials.

\begin{definition}
 Let $F \subseteq \R\langle{X}\rangle_Q$. We call $f \in \R\langle{X}\rangle$ a \emph{$Q$-consequence} of $F$, if it is compatible with $Q$ and there are finitely many polynomials $a_i,b_i \in \R\langle{X}\rangle_Q$ and $f_i \in F$ such that $f=\sum_ia_if_ib_i$ where each summand satisfies $\sigma(a_if_ib_i) \supseteq \sigma(f)$.
\end{definition}

Note that the sum $f=\sum_ia_if_ib_i$ in the definition above can also be empty. So, zero is a $Q$-consequence of any $F \subseteq \R\langle{X}\rangle_Q$, even of $F=\emptyset$, as long as $Q$ has at least one vertex.

From the definition, it easily follows that the sum of two uniformly compatible $Q$-consequences of $F$ is again a $Q$-consequence of $F$ if it is uniformly compatible with $Q$. Furthermore, any $Q$-consequence of some $G \subseteq \R\langle{X}\rangle_Q$ evidently is also a $Q$-consequence of any $F \subseteq \R\langle{X}\rangle_Q$ that contains $G$.
More generally, the property of being a $Q$-consequence is transitive, which we will obtain as a consequence of the theorem below. Proving it directly from the definition would be tedious and technical. 

The main result of this section characterizes the $Q$-consequences of a set of uniformly compatible polynomials $F \subseteq \R\langle{X}\rangle_Q$.
Evidently, all $Q$-consequences of $F$ are contained in the ideal $(F)$ and are compatible with $Q$ by definition. We show the surprising fact that also the converse is true, i.e.\ any element of the ideal $(F)$ that is compatible with $Q$ is already a $Q$-consequence of $F$.
This relies on the key observation of Lemma~\ref{lem:inclusion}.

\begin{theorem}
\label{thm:Qconsequences}
 Let $Q$ be a labelled quiver with labels $X$ and let $F \subseteq \R\langle{X}\rangle_Q$. Then, for any $f \in \R\langle{X}\rangle$ the following are equivalent:
 \begin{enumerate}
  \item\label{item:ideal} $f$ is compatible with $Q$ and lies in the ideal generated by $F$.
  \item\label{item:rewriting} $f$ is compatible with $Q$ and can be rewritten to zero using $F$.
  \item\label{item:consequence} $f$ is a $Q$-consequence of $F$.
 \end{enumerate}
\end{theorem}
\begin{proof}
 By Lemma~\ref{lem:idealmembership}, statements \ref{item:ideal} and \ref{item:rewriting} are equivalent.
 Moreover, we note that statement~\ref{item:consequence} implies statement~\ref{item:ideal} by definition.
 To conclude, we prove statement~\ref{item:consequence} from statement~\ref{item:rewriting}.
 Assume $f \in \R\langle{X}\rangle_Q$ can be rewritten to zero using $F$. This means, there are $h_0,\dots,h_n \in \R\langle{X}\rangle$, with $h_0=f$ and $h_n=0$ such that, for all $i \in \{1,\dots,n\}$, $h_i$ can be obtained from $h_{i-1}$ by a rewriting step using some $g_i \in F$. In particular, this means that there are $\lambda_i \in \R$ and $a_i,b_i \in \langle{X}\rangle$ such that $h_i=h_{i-1}+\lambda_ia_ig_ib_i$. By Lemma~\ref{lem:inclusion}, we infer inductively that all $a_i,b_i,a_ig_ib_i$ are uniformly compatible with $Q$ and satisfy $\sigma(h_i) \supseteq \sigma(h_{i-1}) \supseteq \sigma(f)$ and $\sigma(a_ig_ib_i) \supseteq \sigma(h_{i-1}) \supseteq \sigma(f)$. Therefore, $f=\sum_{i=1}^n(-\lambda_i)a_ig_ib_i$ is a $Q$-consequence of $F$.
\end{proof}

In particular, the above theorem shows that the $Q$-consequences of $F$ are precisely the compatible elements of the ideal generated by $F$.
Moreover, since ideal membership is independent of the quiver, the following statement follows immediately.
\begin{corollary}
\label{cor:consequences}
 Let $F \subseteq \R\langle{X}\rangle$ and $f \in (F)$. Then, for all labelled quivers $Q$ such that all elements of $F$ are uniformly compatible with $Q$ we have that
 \[
  f\text{ is compatible with }Q \quad\Longleftrightarrow\quad f\text{ is a $Q$-consequence of }F.
 \]
\end{corollary}

Furthermore, since $G \subseteq (F)$ implies $(G) \subseteq (F)$, the above theorem also proves that the notion of $Q$-consequences is transitive in the following sense.

\begin{corollary}
 Let $Q$ be a labelled quiver with labels $X$ and let $F \subseteq \R\langle{X}\rangle_Q$. If $G \subseteq \R\langle{X}\rangle_Q$ is a set of $Q$-consequences of $F$, then any $Q$-consequence of $G$ is also a $Q$-consequence of $F$.
\end{corollary}

\section{Paradigm for proving identities of operators}
\label{sec:RlinearOperators}

Now, we come back to the main question: if we have an ideal generated by some set of polynomials $F \subseteq \R\langle{X}\rangle$ along with an element $f \in (F)$ of that ideal, what does that prove when $F$ and $f$ are interpreted in the context of operators?
Based on the framework developed in the previous sections, in particular Theorem~\ref{thm:Qconsequences}, Theorem~\ref{thm:Rlinear} below gives the answer for a very general notion of operators.
First, we present the more concrete case of operators being viewed as linear maps between vector spaces over some field $K$. For appropriate fields $K$, Theorem~\ref{thm:Klinear} already covers many interesting cases such as real or complex matrices or bounded linear maps between Hilbert spaces, for example.
The more general notion of operators considered later uses the more abstract language of $\R$-linear categories. Beyond what can be viewed as linear maps between vector spaces, it allows to treat also matrices over rings, homomorphisms of modules, or elements of an algebra.

\subsection{Linear maps between vector spaces}

For a quiver $(V,E,s,t)$ and a field $K$, $(\mathcal{V},\varphi)$ is called a \emph{representation} of the quiver $(V,E,s,t)$, if $\mathcal{V}=(\mathcal{V}_v)_{v \in V}$ is a family of $K$-vector spaces and $\varphi$ is a map that assigns to each $e \in E$ a $K$-linear map $\varphi(e): \mathcal{V}_{s(e)} \to \mathcal{V}_{t(e)}$, see e.g.\ \cite{DerksenWeyman2005}.
Note that any nonempty path $e_n{\dots}e_1$ in the quiver induces a $K$-linear map $\varphi(e_n){\cdot}{\dots}{\cdot}\varphi(e_1)$, since the maps $\varphi(e_{i+1})$ and $\varphi(e_i)$ can be composed for every $i \in \{1,\dots,n-1\}$ by definition of $\varphi$. Similarly, for every $v \in V$, the empty path $\epsilon_v$ induces the identity map on $\mathcal{V}_v$.

\begin{remark}\label{rem:consistencyUnique}
 If, as in Theorem~\ref{thm:KlinearUnique}, edges have unique labels and only polynomials without constant term are considered, then every compatible polynomial has a unique realization as a $K$-linear map and every representation of the quiver is consistent with its labelling as defined below.
\end{remark}

\begin{definition}
 Let $K$ be a field and let $Q$ be a labelled quiver with labelling $l$. We call a representation $(\mathcal{V},\varphi)$ of $Q$ \emph{consistent} with the labelling $l$ if for any two nonempty paths $p=e_n{\dots}e_1$ and $q=d_n{\dots}d_1$ in $Q$ with the same source and target, equality of labels $l(p)=l(q)$ implies $\varphi(e_n){\cdot}{\dots}{\cdot}\varphi(e_1) = \varphi(d_n){\cdot}{\dots}{\cdot}\varphi(d_1)$ as $K$-linear maps.
\end{definition}

\begin{remark}\label{rem:consistency}
 Trivially, if there are no distinct paths with the same source, traget, and label, then every representation of that labelled quiver is consistent with its labelling. In practice, this is usually fulfilled by quivers originating from statements about operators. The following sufficient conditions can be verified without the need for looking at paths or linear maps.
 If for every vertex $v \in V$ all outgoing edges, i.e.\ $e \in E$ with $s(e)=v$, have distinct labels, then all paths with the same source have distinct labels. Likewise, if for every vertex all incoming edges have distinct labels, then all paths with the same target have distinct labels. 
\end{remark}

For the next definition and lemma, we fix a field $K$, a labelled quiver $Q=(V,E,X,s,t,l)$, and a consistent representation $\mathcal{R}=(\mathcal{V},\varphi)$ of $Q$.
Now, we formalize the intuitive concept of plugging in the $K$-linear maps $\varphi(e)$, $e \in E$, in place of the indeterminates $l(e)$ of polynomials in $K\langle{X}\rangle$.
In general, of course, this can only be done if all monomials in the support of the given polynomial are labels of paths in $Q$ with the same source $v \in V$ and with the same target $w \in V$.
In other words, the polynomials have to lie in the $K$-vector space $K\langle{X}\rangle_{v,w}$, see Definition~\ref{def:module}.
As a result of plugging in linear maps for indeterminates, we obtain an element in $L(\mathcal{V}_v,\mathcal{V}_w)$, the set of $K$-linear maps from $\mathcal{V}_v$ to $\mathcal{V}_w$.
Note that, for fixed source and target, the induced map is independent of the paths chosen, since the representation $\mathcal{R}$ is consistent with the labelling $l$. Consequently, the map $\varphi_{v,w}$ defined below is well-defined.

\begin{definition}
 For $v,w \in V$, we define the $K$-linear map $\varphi_{v,w} : K\langle{X}\rangle_{v,w} \to L(\mathcal{V}_v,\mathcal{V}_w)$ by
 \[
  \varphi_{v,w}(l(e_n{\dots}e_1)) := \varphi(e_n){\cdot}{\dots}{\cdot}\varphi(e_1)
 \]
 for all nonempty paths $e_n{\dots}e_1$ in $Q$ from $v$ to $w$ and, if $v=w$, also by $\varphi_{v,v}(1) := \id_{\mathcal{V}_v}$.
 For all $f \in K\langle{X}\rangle_{v,w}$, we call the $K$-linear map $\varphi_{v,w}(f)$ a \emph{realization} of $f$ w.r.t.\ the representation $\mathcal{R}$ of $Q$.
\end{definition}

Note that, a polynomial is compatible if and only if it has at least one realization as a $K$-linear map.
For any $v,w \in V$, the zero map from $\mathcal{V}_v$ to $\mathcal{V}_w$ is a realization of $0 \in K\langle{X}\rangle$.

\begin{example}
\label{ex:realizations}
 The following three diagrams show a quiver along with a labelling and a representation of it, which can be used to describe the situation of the introductory example. The representation is trivially consistent with the labelling since edges have unique labels, see Remark~\ref{rem:consistency}. We assume $\mathcal{V}_v$ and $\mathcal{V}_w$ are $K$-vector spaces as well as $A \in L(\mathcal{V}_v,\mathcal{V}_w)$ and $A^-,Y \in L(\mathcal{V}_w,\mathcal{V}_v)$.
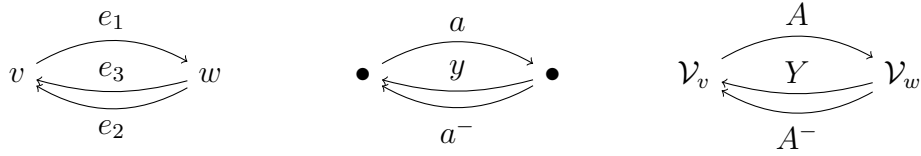
\begin{center}
\begin{tikzpicture}
\begin{scope}[xshift=-4.5cm]
 \matrix (m) [matrix of math nodes, column sep=2cm]
  {v & w \\};
 \path[->] (m-1-1) edge [bend left] node [auto] {$e_1$} (m-1-2);
 \path[->] (m-1-2) edge [bend left] node [auto] {$e_2$} (m-1-1);
 \path[->] (m-1-2) edge [bend left=15] node [auto, swap] {$e_3$} (m-1-1);
\end{scope}
\begin{scope}
 \matrix (m) [matrix of math nodes, column sep=2cm]
  {\bullet & \bullet \\};
 \path[->] (m-1-1) edge [bend left] node [auto] {$a$} (m-1-2);
 \path[->] (m-1-2) edge [bend left] node [auto] {$a^-$} (m-1-1);
 \path[->] (m-1-2) edge [bend left=15] node [auto, swap] {$y$} (m-1-1);
\end{scope}
\begin{scope}[xshift=4.5cm]
 \matrix (m) [matrix of math nodes, column sep=2cm]
  {\mathcal{V}_v & \mathcal{V}_w \\};
 \path[->] (m-1-1) edge [bend left] node [auto] {$A$} (m-1-2);
 \path[->] (m-1-2) edge [bend left] node [auto] {$A^-$} (m-1-1);
 \path[->] (m-1-2) edge [bend left=15] node [auto, swap] {$Y$} (m-1-1);
\end{scope}
\end{tikzpicture}
\end{center}
 Then, the polynomials $g = aa^-a-a$ and $f = a(a^-+y-yaa^-)a-a$ lie in $K\langle{X}\rangle_{v,w}$. So, in $L(\mathcal{V}_v,\mathcal{V}_w)$, these polynomials have the realizations
 \[
  \varphi_{v,w}(g) = \varphi_{v,w}(l(e_1e_2e_1))-\varphi_{v,w}(l(e_1)) = \varphi(e_1){\cdot}\varphi(e_2){\cdot}\varphi(e_1)-\varphi(e_1) = AA^-A-A
 \]
 and similiarly
 \begin{align*}
  \varphi_{v,w}(f) &= \varphi_{v,w}(l(e_1e_2e_1))+\varphi_{v,w}(l(e_1e_3e_1))-\varphi_{v,w}(l(e_1e_3e_1e_2e_1))-\varphi_{v,w}(l(e_1))\\
   &= AA^-A+AYA-AYAA^-A-A.
 \end{align*}
 As discussed before, the same polynomials can also be used for other situations. In particular, we can consider quivers where edges can have the same label. The following labelled quiver has this property and has been mentioned before already. It is obtained by adding a vertex and an edge to the above quiver and labelling the new edge with $a$ as well. With $A_1 \in L(\mathcal{V}_u,\mathcal{V}_w)$ and $A_2 \in L(\mathcal{V}_v,\mathcal{V}_w)$ for some $K$-vector space $\mathcal{V}_u$, the following representation of it is consistent with the labelling, since all edges with the same source have distinct labels, c.f.\ Remark~\ref{rem:consistency}.
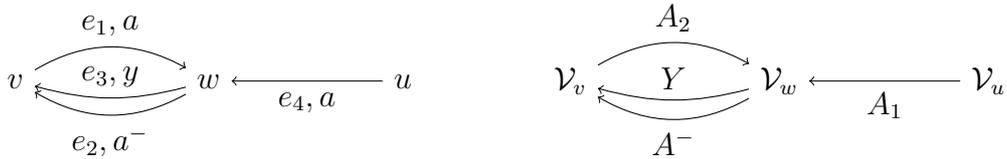
\begin{center}
\begin{tikzpicture}
\begin{scope}
 \matrix (m) [matrix of math nodes, column sep=2cm]
  {v & w & u\\};
 \path[->] (m-1-3) edge node [auto] {$e_4,a$} (m-1-2);
 \path[->] (m-1-1) edge [bend left] node [auto] {$e_1,a$} (m-1-2);
 \path[->] (m-1-2) edge [bend left] node [auto] {$e_2,a^-$} (m-1-1);
 \path[->] (m-1-2) edge [bend left=15] node [auto, swap] {$e_3,y$} (m-1-1);
\end{scope}
\begin{scope}[xshift=7.5cm]
 \matrix (m) [matrix of math nodes, column sep=2cm]
  {\mathcal{V}_v & \mathcal{V}_w & \mathcal{V}_u\\};
 \path[->] (m-1-3) edge node [auto] {$A_1$} (m-1-2);
 \path[->] (m-1-1) edge [bend left] node [auto] {$A_2$} (m-1-2);
 \path[->] (m-1-2) edge [bend left] node [auto] {$A^-$} (m-1-1);
 \path[->] (m-1-2) edge [bend left=15] node [auto, swap] {$Y$} (m-1-1);
\end{scope}
\end{tikzpicture}
\end{center}
 Now, the polynomials $f$ and $g$ lie not only in $K\langle{X}\rangle_{v,w}$ but also in $K\langle{X}\rangle_{u,w}$. So, each of them has realizations in $L(\mathcal{V}_v,\mathcal{V}_w)$ and in $L(\mathcal{V}_u,\mathcal{V}_w)$. For instance, the realizations of $g$ are given by:
 \begin{align*}
  \varphi_{u,w}(g) &= \varphi_{u,w}(l(e_1e_2e_4))-\varphi_{u,w}(l(e_4)) = A_2A^-A_1-A_1\\
  \varphi_{v,w}(g) &= \varphi_{v,w}(l(e_1e_2e_1))-\varphi_{v,w}(l(e_1)) = A_2A^-A_2-A_2.
 \end{align*}
\end{example}

By definition, for any $K$-linear combination of compatible polynomials $f_i$ that share a signature $(v,w)$, we have that the realization of the linear combination in $L(\mathcal{V}_v,\mathcal{V}_w)$ is the corresponding linear combination of realizations of the $f_i$.
Similarly, for the product of compatible polynomials with matching source and target, we have that the corresponding realization of the product is the composition of realizations of the factors, for the precise statement see the lemma below.
Altogether, it follows from Lemma~\ref{lem:uniform} that computations inside the set $K\langle{X}\rangle_Q$ of uniformly compatible polynomials can also be done with their realizations.

\begin{lemma}\label{lem:productrealizationsK}
 Let $u,v,w \in V$. Then, for all $f \in K\langle{X}\rangle_{v,w}$ and $g \in K\langle{X}\rangle_{u,v}$, we have that $fg \in K\langle{X}\rangle_{u,w}$ and
 \[
  \varphi_{u,w}(fg) = \varphi_{v,w}(f){\cdot}\varphi_{u,v}(g).
 \]
\end{lemma}
\begin{proof}
 Since the three maps $\varphi_{u,w},\varphi_{v,w},\varphi_{u,v}$ are $K$-linear and composition of linear maps $L(\mathcal{V}_v,\mathcal{V}_w) \times L(\mathcal{V}_u,\mathcal{V}_v) \to L(\mathcal{V}_u,\mathcal{V}_w)$ is $K$-bilinear, it suffices to show the statement for $f$ and $g$ being monomials.
 Let $p=e_n{\dots}e_{k+1}$ be a nonempty path in $Q$ with source $v$ and target $w$ and $f=l(p)$. Likewise, let $q=e_k{\dots}e_1$ be a nonempty path in $Q$ with source $u$ and target $v$ and $g=l(q)$. Then, $fg \in K\langle{X}\rangle_{u,w}$ obviously holds. Moreover,
 \[
  \varphi_{u,w}(fg) = \varphi_{u,w}(l(pq)) = \varphi(e_n){\cdot}{\dots}{\cdot}\varphi(e_{k+1}){\cdot}\varphi(e_k){\cdot}{\dots}{\cdot}\varphi(e_1) = \varphi_{v,w}(f){\cdot}\varphi_{u,v}(g).
 \]
 If $p$ or $q$ are empty, then $\varphi_{u,w}(l(pq)) = \varphi_{v,w}(f){\cdot}\varphi_{u,v}(g)$ is trivially true.
\end{proof}

Recall that $Q$-consequences have representations for which every intermediate expression is compatible with $Q$ and satisfies additional conditions. These representations can be carried over to all realizations of the $Q$-consequence by the properties of (uniformly) compatible polynomials and their realizations. Therefore, we obtain the following theorem, since every compatible polynomial in the ideal is a $Q$-consequence of the generators by Theorem~\ref{thm:Qconsequences}.

\begin{theorem}\label{thm:Klinear}
 Let $\R=K$ be a field and let $F \subseteq K\langle{X}\rangle$ and $f \in (F)$. Moreover, let $Q=(V,E,X,s,t,l)$ be a labelled quiver such that $f$ is compatible and all elements of $F$ are uniformly compatible with $Q$. Then, for all consistent representations $\mathcal{R}=(\mathcal{V},\varphi)$ of $Q$ such that every realization of any element of $F$ w.r.t.\ $\mathcal{R}$ is zero, we have that every realization of $f$ w.r.t.\ $\mathcal{R}$ is zero.
\end{theorem}

\begin{proof}
 If $f=0$, then the statement trivially holds, so we assume $f\neq0$.
 Let $\varphi_{v,w}(f)$ be a realization of $f$, then $(v,w) \in \sigma(f)$ and $\varphi_{v,w}(f) \in L(\mathcal{V}_v,\mathcal{V}_w)$.
 Since $f$ is compatible with $Q$ and $F \subseteq K\langle{X}\rangle_Q$, $f$ is a $Q$-consequence of $F$ by Theorem~\ref{thm:Qconsequences}.
 This means, there are finitely many polynomials $a_i,b_i \in K\langle{X}\rangle_Q$ and $f_i \in F$ such that $f=\sum_{i=1}^na_if_ib_i$ where each summand satisfies $\sigma(a_if_ib_i) \supseteq \sigma(f)$. In particular, all summands $a_if_ib_i$ are contained in $K\langle{X}\rangle_{v,w}$. Without loss of generality, we assume that all summands $a_if_ib_i$ are nonzero.
 For every $i \in \{1,\dots,n\}$, by applying Lemma~\ref{lem:multiplication} the product $a_if_i$ is uniformly compatible. Hence, by applying Lemma~\ref{lem:multiplication} and Remark~\ref{rem:factorization} twice, there are $v_i,w_i \in V$ such that $a_i \in K\langle{X}\rangle_{w_i,w}$, $f_i \in K\langle{X}\rangle_{v_i,w_i}$, and $b_i \in K\langle{X}\rangle_{v,v_i}$. Then, by Lemma~\ref{lem:productrealizationsK}, $\varphi_{v,w}(a_if_ib_i) = \varphi_{w_i,w}(a_i){\cdot}\varphi_{v_i,w_i}(f_i){\cdot}\varphi_{v,v_i}(b_i)$.
 Consequently, we compute
 \[
  \varphi_{v,w}(f) = \sum_{i=1}^n\varphi_{v,w}(a_if_ib_i) = \sum_{i=1}^n\varphi_{w_i,w}(a_i){\cdot}\varphi_{v_i,w_i}(f_i){\cdot}\varphi_{v,v_i}(b_i)
 \]
 by additivity of $\varphi_{v,w}$. Since $\varphi_{v_i,w_i}(f_i)$ is zero for all $i \in \{1,\dots,n\}$ by assumption, we conclude that $\varphi_{v,w}(f)$ is zero as well.
\end{proof}

\begin{example}
\label{ex:theorem}
 We illustrate how the above theorem can be used to prove the following statement mentioned in the introduction: given $A \in L(\mathcal{V}_v,\mathcal{V}_w)$ and an inner inverse $A^- \in L(\mathcal{V}_w,\mathcal{V}_v)$, then also $A^-+Y-YAA^-$ is an inner inverse of $A$ for any $Y \in L(\mathcal{V}_w,\mathcal{V}_v)$.
 It already has been checked in the introduction that the polynomial $f = a(a^-+y-yaa^-)a-a$ lies in the ideal generated by $g = aa^-a-a$, since we have the representation $f = (1-ay)g$. We take the quiver and its consistent representation described in the first part of Example~\ref{ex:realizations}. Uniform compatibility of $f$ and $g$ with the quiver has been checked in Example~\ref{ex:compatible}. Since $A^-$ is an inner inverse of $A$ by assumption, the realization $AA^-A-A$ of $g$ computed in Example~\ref{ex:realizations} is zero.
 Then, Theorem~\ref{thm:Klinear} implies that the realization $A(A^-+Y-YAA^-)A-A$ of $f$ is zero. Hence, $A^-+Y-YAA^-$ is an inner inverse of $A$.
 Note that, using other representations or other quivers, analogous or more general statements can be proven based on the single formal computation $f = (1-ay)g$.
\end{example}

In order to apply our main theorem, the assumptions on the operators involved need to be expressed as identities of operators first. It is not always immediately obvious how to do so and may involve additional operators, as the following example shows.
The example also explicitly shows how analogous results are obtained, without redoing the computation, just by changing the representation of the quiver.

\begin{example}
\label{ex:WoodburyInner}
 In \cite{HendersonSearle1981}, related to the Sherman-Morrison-Woodbury formula, several formulae for the inverse of the sum of two matrices are reviewed. Also a few formulae for inner inverses are given there, of which we treat one now.
 Let $A \in \mathbb{C}^{m\times{n}}$, $B \in \mathbb{C}^{p\times{q}}$, $Y \in \mathbb{C}^{m\times{p}}$, and $Z \in \mathbb{C}^{q\times{n}}$ such that the column space and the row space of $YBZ$ are contained in the column space and row space of $A$, respectively. If $A^- \in \mathbb{C}^{n\times{m}}$ and $S^- \in \mathbb{C}^{q\times{p}}$ are inner inverses of $A$ and $S:=B+BZA^-YB$, respectively, then the claim is that
 \[
  A^--A^-YBS^-BZA^-
 \]
 is an inner inverse of $A+YBZ$.
 With the aid of $A^-$, the assumptions on the spaces of $YBZ$ can be expressed as identities $AA^-YBZ=YBZ$ and $YBZA^-A=YBZ$.
 In $\mathbb{Z}\langle{a,a^-,b,s^-,y,z}\rangle$, the assumptions are represented by the following polynomials.
 \begin{gather*}
  f_1 = aa^-a-a, \quad\quad f_2 = aa^-ybz-ybz, \quad\quad f_3 = ybza^-a-ybz,\\
  f_4 = (b+bza^-yb)s^-(b+bza^-yb)-(b+bza^-yb)
 \end{gather*}
 It can be checked that the polynomial
 \[
  f = (a+ybz)(a^--a^-ybs^-bza^-)(a+ybz)-(a+ybz),
 \]
 which represents the claim, can be expressed in terms of $f_1,f_2,f_3,f_4$ as
 \begin{multline*}
  f = f_1 + f_2(1+(a^-ybs^-bz-1)a^-(a+ybz))\\
   + aa^-y((1+bza^-y)bs^--1)bza^-f_3 - aa^-yf_4za^-a.
 \end{multline*}
 Since this expression uses only integer coefficients, we have that $f \in (f_1,f_2,f_3,f_4)$ in $K\langle{a,a^-,b,s^-,y,z}\rangle$ for any field $K$, in particular $K=\mathbb{C}$.
 Furthermore, it can be checked that the polynomials $f_1,f_2,f_3,f_4,f$ are uniformly compatible with the labelled quiver shown below on the left. Since the labelling of the quiver does not label different edges the same way, any representation of the quiver is consistent with the labelling, see Remark~\ref{rem:consistency}. If we choose the representation shown below on the right with $\mathcal{V}_1=\mathbb{C}^n$, $\mathcal{V}_2=\mathbb{C}^m$, $\mathcal{V}_3=\mathbb{C}^q$, and $\mathcal{V}_4=\mathbb{C}^p$, then Theorem~\ref{thm:Klinear} allows us to conclude that under the assumptions above every realization of the polynomial $f$ w.r.t.\ this representation of the quiver is zero, i.e.\ $A^--A^-YBS^-BZA^-$ is indeed an inner inverse of $A+YBZ$.
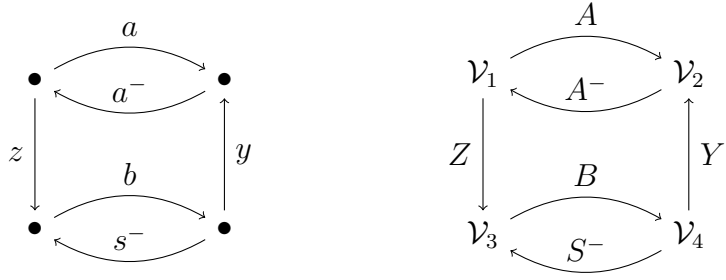
\begin{center}
\begin{tikzpicture}
\begin{scope}
 \matrix (m) [matrix of math nodes, column sep=2cm, row sep=1.5cm]
  {\bullet & \bullet \\
  \bullet & \bullet\\};
 \path[->] (m-1-1) edge [bend left] node [auto] {$a$} (m-1-2);
 \path[->] (m-1-2) edge [bend left] node [auto, swap] {$a^-$} (m-1-1);
 \path[->] (m-2-1) edge [bend left] node [auto] {$b$} (m-2-2);
 \path[->] (m-2-2) edge node [auto, swap] {$y$} (m-1-2);
 \path[->] (m-1-1) edge node [auto, swap] {$z$} (m-2-1);
 \path[->] (m-2-2) edge [bend left] node [auto, swap] {$s^-$} (m-2-1);
\end{scope}
\begin{scope}[xshift=6cm]
 \matrix (m) [matrix of math nodes, column sep=2cm, row sep=1.5cm]
  {\mathcal{V}_1 & \mathcal{V}_2 \\
  \mathcal{V}_3 & \mathcal{V}_4\\};
 \path[->] (m-1-1) edge [bend left] node [auto] {$A$} (m-1-2);
 \path[->] (m-1-2) edge [bend left] node [auto, swap] {$A^-$} (m-1-1);
 \path[->] (m-2-1) edge [bend left] node [auto] {$B$} (m-2-2);
 \path[->] (m-2-2) edge node [auto, swap] {$Y$} (m-1-2);
 \path[->] (m-1-1) edge node [auto, swap] {$Z$} (m-2-1);
 \path[->] (m-2-2) edge [bend left] node [auto, swap] {$S^-$} (m-2-1);
\end{scope}
\end{tikzpicture}
\end{center}
 Choosing other representations of this quiver also allows us to prove analogous statements in other contexts. For example, let $\mathcal{V}_1,\mathcal{V}_2,\mathcal{V}_3,\mathcal{V}_4$ be arbitrary Banach spaces and let $A,A^-,B,S^-,Y,Z$ be bounded linear operators between them such that $\im(YBZ) \subseteq \im(A)$, $\ker(YBZ) \supseteq \ker(A)$, and $A^-$ resp.\ $S^-$ are inner inverses of $A$ resp.\ $S:=B+BZA^-YB$. Note that, because of $AA^-A=A$, the conditions on range and kernel of $YBZ$ are equivalent to the identities $AA^-YBZ=YBZ$ and $YBZA^-A=YBZ$, respectively. Then, based on ideal membership and compatibility checked above, Theorem~\ref{thm:Klinear} implies that the bounded linear operator $A^--A^-YBS^-BZA^-$ is an inner inverse of $A+YBZ$.
\end{example}

If a statement involves several versions of the same operator, our framework allows to assign the same label to the corresponding edges of the quiver provided that these versions of the operator satisfy the same identities. This is because the polynomials representing the assumptions have to be uniformly compatible with the quiver.
For instance, differentiation and integration can act on various classes of functions, where they always satisfy the fundamental theorem of calculus.
The following example illustrates this using a solution formula for reducible second-order linear ordinary differential equations.

\begin{example}
\label{ex:ODEfact}
 Consider the inhomogeneous linear differential equation
 \[
  y^{\prime\prime}(x)+A_1(x)y^\prime(x)+A_0(x)y(x)=r(x)
 \]
 and assume that it can be decomposed into the two first-order equations
 \[
  y^\prime(x)-B_2(x)y(x)=z(x) \quad\text{and}\quad z^\prime(x)-B_1(x)z(x)=r(x).
 \]
 We use our framework to show that a particular solution is given by the nested integral
 \[
  y(x) = H_2(x)\int_{x_2}^xH_2(t)^{-1}H_1(t)\int_{x_1}^tH_1(u)^{-1}r(u)\,du\,dt,
 \]
 where $H_i(x)$ is a solution of $y^\prime(x)-B_i(x)y(x)=0$ such that $H_i(x)^{-1}$ exists.
 In terms of operators, we have to consider differentiation $\Der$ mapping $y(x)$ to $y'(x)$ and integration $\Int_1$ resp.\ $\Int_2$ mapping $y(x)$ to $\int_{x_1}^xy(t)\,dt$ resp.\ $\int_{x_2}^xy(t)\,dt$. Based on the fundamental theorem of calculus, these operators satisfy the identities $\Der{\cdot}\Int_1=\id$ and $\Der{\cdot}\Int_2=\id$. Any function $F(x)$ induces a multiplication operator denoted by $F$, which maps $y(x)$ to $F(x)y(x)$. By the Leibniz rule, differentiation and multiplication operators satisfy the identity $\Der{\cdot}F=F{\cdot}\Der+F^\prime$. In particular, for the multiplication operators $H_i$ we obtain $\Der{\cdot}H_i=H_i{\cdot}\Der+B_i{\cdot}H_i$ since the functions $H_i(x)$ satisfy $H_i^\prime(x)=B_i(x)H_i(x)$. The operators corresponding to the differential equation and the solution formula are given by
 \[
  L:=\Der^2+A_1{\cdot}\Der+A_0 \quad\text{and}\quad S:=H_2{\cdot}\Int_2{\cdot}H_2^{-1}{\cdot}H_1{\cdot}\Int_1{\cdot}H_1^{-1}
 \]
 and we need to show that $L{\cdot}S=\id$ provided the decomposition $L = (\Der-B_1){\cdot}(\Der-B_2)$ holds. In $\mathbb{Z}\langle{a_0,a_1,b_1,b_2,d,h_1,\tilde{h}_1,h_2,\tilde{h}_2,i}\rangle$, we use $d$ and $i$ to represent differentiation and integration. Then, the properties of the operators are represented by the following polynomials.
 \begin{gather*}
  f_1 = d^2+a_1d+a_0-(d-b_1)(d-b_2), \quad\quad f_2 = di-1,\\
  f_3 = d h_1-h_1d-b_1h_1, \quad\quad f_4 = d h_2-h_2d-b_2h_2,\\
  f_5 = h_1\tilde{h}_1-1, \quad\quad f_6 = h_2\tilde{h}_2-1
 \end{gather*}
 The polynomial
 \[
  f = (d^2+a_1d+a_0)h_2i\tilde{h}_2h_1i\tilde{h}_1-1
 \]
 lies in the ideal $(f_1,f_2,f_3,f_4,f_5,f_6)$, as can be seen from the following representation.
 \begin{multline*}
  f = f_1h_2i\tilde{h}_2h_1i\tilde{h}_1 + (d-b_1)h_2f_2\tilde{h}_2h_1i\tilde{h}_1 + h_1f_2\tilde{h}_1\\
   + f_3i\tilde{h}_1 + (d-b_1)f_4i\tilde{h}_2h_1i\tilde{h}_1 + f_5 + (d-b_1)f_6h_1i\tilde{h}_1
 \end{multline*}
 Based on our main theorem, this ideal membership gives rise to several statements about actual functions and operators.
 We assume that the functions inducing the multiplication operators are infinitely differentiable, i.e.\ $A_0,A_1,B_1,B_2,H_1,H_1^{-1},H_2,H_2^{-1} \in C^\infty(I)$ for some interval $I \subseteq \mathbb{R}$. Since differentiation reduces the smoothness of functions, it is natural to consider the operators $\Der : C^{k+1}(I) \to C^{k}(I)$ for all $k \in \mathbb{N}$. Similarly, for integration we consider the operators $\Int_k : C^{k}(I) \to C^{k+1}(I)$ mapping $y(x)$ to $\int_{x_k}^xy(t)\,dt$ for some $x_k \in I$. This leads us to consider the following infinite quiver.
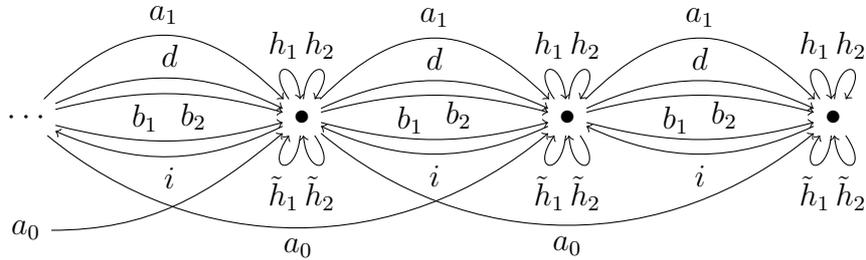
\begin{center}
\begin{tikzpicture}
\begin{scope}
 \matrix (m) [matrix of math nodes, column sep=3cm, row sep=1cm]
  {\cdots & \bullet & \bullet & \bullet\\
  a_0&\mbox{}&\mbox{}&\mbox{}\\};
 \path[->] (m-1-1) edge [bend left=45,looseness=1.3] node [above] {$a_1$} (m-1-2);
 \path[->] (m-1-2) edge [bend left=45,looseness=1.3] node [above] {$a_1$} (m-1-3);
 \path[->] (m-1-3) edge [bend left=45,looseness=1.3] node [above] {$a_1$} (m-1-4);
 \path[->] (m-1-1) edge [bend left=25] node [above] {$d$} (m-1-2);
 \path[->] (m-1-2) edge [bend left=25] node [above] {$d$} (m-1-3);
 \path[->] (m-1-3) edge [bend left=25] node [above] {$d$} (m-1-4);
 \path[->] (m-1-1) edge [bend left=15] node [below left] {$b_1$} (m-1-2);
 \path[->] (m-1-2) edge [bend left=15] node [below left] {$b_1$} (m-1-3);
 \path[->] (m-1-3) edge [bend left=15] node [below left] {$b_1$} (m-1-4);
 \path[->] (m-1-1) edge [bend right=15] node [above right] {$b_2$} (m-1-2);
 \path[->] (m-1-2) edge [bend right=15] node [above right] {$b_2$} (m-1-3);
 \path[->] (m-1-3) edge [bend right=15] node [above right] {$b_2$} (m-1-4);
 \path[->] (m-1-2) edge [bend left=25] node [below] {$i$} (m-1-1);
 \path[->] (m-1-3) edge [bend left=25] node [below] {$i$} (m-1-2);
 \path[->] (m-1-4) edge [bend left=25] node [below] {$i$} (m-1-3);
 \path[->] (m-2-1) edge [out=0,in=220] (m-1-2);
 \path[->] (m-1-1) edge [bend right=40] node [below] {$a_0$} (m-1-3);
 \path[->] (m-1-2) edge [bend right=40] node [below] {$a_0$} (m-1-4);
 \path[->] (m-1-2) edge [out=125,in=95,looseness=10] node [above] {$h_1$} (m-1-2);
 \path[->] (m-1-3) edge [out=125,in=95,looseness=10] node [above] {$h_1$} (m-1-3);
 \path[->] (m-1-4) edge [out=125,in=95,looseness=10] node [above] {$h_1$} (m-1-4);
 \path[->] (m-1-2) edge [out=85,in=55,looseness=10] node [above] {$h_2$} (m-1-2);
 \path[->] (m-1-3) edge [out=85,in=55,looseness=10] node [above] {$h_2$} (m-1-3);
 \path[->] (m-1-4) edge [out=85,in=55,looseness=10] node [above] {$h_2$} (m-1-4);
 \path[->] (m-1-2) edge [out=265,in=235,looseness=10] node [below] {$\tilde{h}_1$} (m-1-2);
 \path[->] (m-1-3) edge [out=265,in=235,looseness=10] node [below] {$\tilde{h}_1$} (m-1-3);
 \path[->] (m-1-4) edge [out=265,in=235,looseness=10] node [below] {$\tilde{h}_1$} (m-1-4);
 \path[->] (m-1-2) edge [out=305,in=275,looseness=10] node [below] {$\tilde{h}_2$} (m-1-2);
 \path[->] (m-1-3) edge [out=305,in=275,looseness=10] node [below] {$\tilde{h}_2$} (m-1-3);
 \path[->] (m-1-4) edge [out=305,in=275,looseness=10] node [below] {$\tilde{h}_2$} (m-1-4);
\end{scope}
\end{tikzpicture}
\end{center}
 Then, it can be checked that the polynomial $f$ and all the $f_i$ are uniformly compatible. Note that each compatible monomial has infinitely many signatures, which can be parameterized. For example, $f_2$ is uniformly compatible since $\sigma(di) = \{(v,v)\ |\ v \in V\} = \sigma(1)$.
 From right to left, we assign the spaces $C(I),C^1(I),C^2(I),\dots$ to the vertices. Since every realization of the $f_i$ is zero, Theorem~\ref{thm:Klinear} implies that also every realization of $f$ is zero. This means that for any $k \in \mathbb{N}$ and $r \in C^k(I)$ the function
 \[
  y(x) = H_2{\cdot}\Int_{k+1}{\cdot}H_2^{-1}{\cdot}H_1{\cdot}\Int_k{\cdot}H_1^{-1}r(x)
 \]
 in $C^{k+2}(I)$ satisfies the inhomogeneous differential equation $Ly(x)=r(x)$, if $H_i^\prime(x)=B_i(x)H_i(x)$.

 Alternatively, we can also consider the spaces $L^1_\mathrm{loc}(I),W^{1,1}_\mathrm{loc}(I),W^{2,1}_\mathrm{loc}(I),\dots$ with derivations $\Der : W^{k+1,1}_\mathrm{loc}(I) \to W^{k,1}_\mathrm{loc}(I)$ for all $k$ to obtain the same statement for $r \in W^{k,1}_\mathrm{loc}(I)$ and any $k$.
 Moreover, we can also consider the matrix case of all the above versions of the statement. For any $n \in \mathbb{N}^+$, let $A_0,A_1,B_1,B_2,H_1,H_1^{-1},H_2,H_2^{-1} \in C^\infty(I)^{n\times{n}}$ such that $H_i^\prime(x)=B_i(x)H_i(x)$. Then, by the computation with noncommutative polynomials above, Theorem~\ref{thm:Klinear} implies that the operator expression above for $y(x)$ indeed gives a solution vector of the system of differential equations $Ly(x)=r(x)$ with $r \in C^k(I)^n$, or $r \in W^{k,1}_\mathrm{loc}(I)^n$, if the system can be decomposed into the two first-order systems $y^\prime(x)-B_2(x)y(x)=z(x)$ and $z^\prime(x)-B_1(x)z(x)=r(x)$. To this end, the representations of the quivers just have to assign the spaces $C(I)^n,C^1(I)^n,C^2(I)^n,\dots$ or $L^1_\mathrm{loc}(I)^n,W^{1,1}_\mathrm{loc}(I)^n,W^{2,1}_\mathrm{loc}(I)^n,\dots$ to the vertices along with corresponding operators for the edges.
\end{example}

\subsection{Morphisms in linear categories}

Not all linear operators are necessarily linear maps on vector spaces. More generally, they can be understood as morphisms in $\R$-linear categories for some commutative ring $\R$ with unit element.
The above paradigm generalizes to this more abstract notion of linear operators in a straightforward manner. Nevertheless, we make the formal statements explicit.
Since we want to make the presentation self-contained, we briefly recall the basic definitions and terminology for categories first.

A category $\mathcal{C}$ consists of a class of objects and a class of morphisms between them, i.e.\ each morphism has a source object and a target object. Objects and morphisms in $\mathcal{C}$ are sometimes also referred to as $\mathcal{C}$-objects and $\mathcal{C}$-morphisms, respectively.
Note that the words \emph{object} and \emph{morphism} do not imply anything about the nature of the elements of these classes here. Intuitively, however, one can think of objects as sets and of morphisms as maps between those sets.
For any two objects $V,W$ in $\mathcal{C}$, the class of morphisms in $\mathcal{C}$ with source $V$ and target $W$ is denoted by $\Hom_{\mathcal{C}}(V,W)$. For any three objects $U,V,W$ in $\mathcal{C}$, any morphisms $\alpha \in \Hom_{\mathcal{C}}(U,V)$ and $\beta \in \Hom_{\mathcal{C}}(V,W)$ can be composed to yield a morphism $\beta{\cdot}\alpha \in \Hom_{\mathcal{C}}(U,W)$ and there exists a morphism $1_V \in \Hom_{\mathcal{C}}(V,V)$ such that, for all such $\alpha$ and $\beta$, we have $1_V{\cdot}\alpha=\alpha$ and $\beta{\cdot}1_V=\beta$. Furthermore, composition of morphisms is associative.
A category $\mathcal{C}$ is called \emph{$\R$-linear} if, for any two objects $V,W$ in $\mathcal{C}$, $\Hom_{\mathcal{C}}(V,W)$ is an $\R$-module and, for any three objects $U,V,W$ of $\mathcal{C}$, composing morphisms from $\Hom_{\mathcal{C}}(U,V)$ and $\Hom_{\mathcal{C}}(V,W)$ is $\R$-bilinear.
In particular, $\mathbb{Z}$-linear categories are usually referred to as pre-additive categories.

For example, any category $\mathcal{C}$ whose objects are abelian groups and whose morphisms are homomorphisms of those groups is pre-additive, if every $\Hom_{\mathcal{C}}(G,H)$ contains the zero map and is closed under pointwise addition and pointwise negation of maps. Also, every (associative but not necessarily commutative) $\R$-algebra with unit element can be viewed as an $\R$-linear category, where the class of objects has only one element and the class of morphisms is the set of elements of the algebra.

In addition to the category of $K$-vector spaces, representations of quivers have also been considered for general categories $\mathcal{C}$, see e.g.\ \cite{GothenKing2005}. A \emph{representation} $\mathcal{R}=(\mathcal{V},\varphi)$ of a quiver $(V,E,s,t)$ in $\mathcal{C}$ assigns to each vertex $v \in V$ of the quiver an object $\mathcal{V}_v$ of the category $\mathcal{C}$ and to each edge $e \in E$ a morphism $\varphi(e) \in \Hom_{\mathcal{C}}(\mathcal{V}_{s(e)},\mathcal{V}_{t(e)})$.

Like before, we call a representation of a labelled quiver \emph{consistent} with the labelling, if each monomial can be realized in at most one way as a morphism, once source and target of the underlying path is fixed.
Consistent representations make the maps $\varphi_{v,w}$ defined below well-defined.
Note that it is still allowed to have several realizations of the same monomial in one $\Hom_{\mathcal{C}}(\mathcal{V}_v,\mathcal{V}_w)$, if the same objects $\mathcal{V}_v$ or $\mathcal{V}_w$ are assigned to more than one vertex and there are paths from/to these vertices with the same label.
\begin{definition}
 Let $\mathcal{C}$ be a category and let $Q$ be a labelled quiver with labelling $l$. We call a representation $(\mathcal{V},\varphi)$ of $Q$ \emph{consistent} with the labelling $l$ if for any two nonempty paths $p=e_n{\dots}e_1$ and $q=d_n{\dots}d_1$ in $Q$ with the same source and target, equality of labels $l(p)=l(q)$ implies $\varphi(e_n){\cdot}{\dots}{\cdot}\varphi(e_1) = \varphi(d_n){\cdot}{\dots}{\cdot}\varphi(d_1)$ as morphisms in $\mathcal{C}$.
\end{definition}

For the following definition and lemma, we fix a commutative ring $\R$ with unit element, an $\R$-linear category $\mathcal{C}$, a labelled quiver $Q=(V,E,X,s,t,l)$, and a consistent representation $\mathcal{R}=(\mathcal{V},\varphi)$ of $Q$ in $\mathcal{C}$.
We now show how computations with polynomials in $\R\langle{X}\rangle$ can be interpreted in terms of computations with morphisms in $\mathcal{C}$.

\begin{definition}
 For $v,w \in V$, we define the $\R$-linear map $\varphi_{v,w} : \R\langle{X}\rangle_{v,w} \to \Hom_{\mathcal{C}}(\mathcal{V}_v,\mathcal{V}_w)$ by
 \[
  \varphi_{v,w}(l(e_n{\dots}e_1)) := \varphi(e_n){\cdot}{\dots}{\cdot}\varphi(e_1)
 \]
 for all nonempty paths $e_n{\dots}e_1$ in $Q$ from $v$ to $w$ and, if $v=w$, also by $\varphi_{v,v}(1) := 1_{\mathcal{V}_v}$.
 For all $f \in \R\langle{X}\rangle_{v,w}$, we call the morphism $\varphi_{v,w}(f)$ a \emph{realization} of $f$ w.r.t.\ the representation $\mathcal{R}$ of $Q$.
\end{definition}

The following lemma and theorem generalize Lemma~\ref{lem:productrealizationsK} and Theorem~\ref{thm:Klinear} to $\R$-linear categories. Since composition of morphisms in $\R$-linear categories is $\R$-bilinear, the proofs are completely analogous and hence we omit them.

\begin{lemma}\label{lem:productrealizations}
 Let $u,v,w \in V$. Then, for all $f \in \R\langle{X}\rangle_{v,w}$ and $g \in \R\langle{X}\rangle_{u,v}$, we have that $fg \in \R\langle{X}\rangle_{u,w}$ and
 \[
  \varphi_{u,w}(fg) = \varphi_{v,w}(f){\cdot}\varphi_{u,v}(g).
 \]
\end{lemma}

\begin{theorem}\label{thm:Rlinear}
 Let $\R$ be a commutative ring with unit element and let $F \subseteq \R\langle{X}\rangle$ and $f \in (F)$. Moreover, let $Q=(V,E,X,s,t,l)$ be a labelled quiver such that $f$ is compatible and all elements of $F$ are uniformly compatible with $Q$. Then, for all $\R$-linear categories $\mathcal{C}$ and all consistent representations $\mathcal{R}=(\mathcal{V},\varphi)$ of $Q$ in $\mathcal{C}$ such that every realization of any element of $F$ w.r.t.\ $\mathcal{R}$ is zero, we have that every realization of $f$ w.r.t.\ $\mathcal{R}$ is zero.
\end{theorem}

\begin{example}
 Analogously to Example~\ref{ex:theorem}, based on the above theorem, the following very general version of the statement mentioned in the introduction can be proven: for any $\R$-linear category $\mathcal{C}$, given a morphism $A \in \Hom_\mathcal{C}(\mathcal{V}_v,\mathcal{V}_w)$ and an inner inverse $A^- \in \Hom_\mathcal{C}(\mathcal{V}_w,\mathcal{V}_v)$, then also $A^-+Y-YAA^-$ is an inner inverse of $A$ for any morphism $Y \in \Hom_\mathcal{C}(\mathcal{V}_w,\mathcal{V}_v)$. In particular, it holds for matrices with entries in $\R$ and for $\R$-module homomorphisms.
 Likewise, using the above theorem, more general versions of the statement mentioned in Example~\ref{ex:WoodburyInner} follow from the ideal membership and compatibility checked there just by changing the representation of the quiver. For instance, a version for $\R$-modules $\mathcal{V}_1,\mathcal{V}_2,\mathcal{V}_3,\mathcal{V}_4$ and $\R$-module homomorpisms $A,A^-,B,S^-,Y,Z$ between them follows.
\end{example}

\section{Conclusion}

For proving operator identities, the framework developed in this paper allows to do a single formal computation in the free algebra $\R\langle{X}\rangle$ with symbols instead of operators. This computation does not have to respect the restrictions imposed by the domains and codomains of operators. Based on ideal membership of noncommutative polynomials, the main theorem immediately implies corresponding operator identities, as long as the polynomials corresponding to assumptions and claims have realizations as operators.
Moreover, by choosing different quivers and representations in different $\R$-linear categories, one ideal membership rigorously proves analogous identities for various settings. In particular, if the polynomials describing the assumptions and claims involve only integer coefficients and ideal membership holds in $\mathbb{Z}\langle{X}\rangle$, like in all examples of this paper, then any $\R$-linear category can be used since the ideal membership then also holds in any $\R\langle{X}\rangle$.

While the computations needed for showing ideal membership in the illustrative examples contained in this paper are rather short and could still be done by hand, in general, the representations of polynomials in terms of the generators of the ideal can be very involved and hard to find.
For instance, in our joint work in progress with Dragana Cvetkovi\'c-Ili\'c and her research group on algebraic proofs for generalized inverses \cite{CvetkovicWei2017}, we observed representations of several hundred terms for the proof of some ideal memberships. Using our framework, checking compatibility of each step in such long computations is not necessary and checking compatibility of polynomials corresponding to assumptions and claims is almost immediate.

In contrast to allowing unrestricted computations with noncommutative polynomials, like in the present paper, we are currently considering certain restrictions on the computations in $\R\langle{X}\rangle$ in order to relax the uniform compatibility imposed on the polynomials corresponding to the assumptions. These results will be part of a future publication.
Another direction to generalize the framework presented here aims at including additional computational steps related to factorization of polynomials, beyond ideal membership. Such steps can be used to model injectivity and surjectivity of operators, for instance.

\section*{Acknowledgements}

The results of this paper were motivated by discussions in the framework of the OeAD project SRB 05/2016 {\lq\lq}Generalized inverses, symbolic computations and operator algebras{\rq\rq} with Dragana Cvetkovi\'c-Ili\'c, Anja Korporal, Jovana Milo\v{s}evi\'c, Marko Petkovi\'c, and Milan Tasi\'c.
The authors were supported by the Austrian Science Fund (FWF): P~27229, P~31952, and P~32301.
The authors would like to thank Cyrille Chenavier, Dragana Cvetkovi\'c-Ili\'c, Clemens Hofstadler, and Markus Passenbrunner for feedback on earlier versions of the manuscript.

\end{document}